\documentclass[12pt]{article}
\usepackage{amsmath}
\usepackage{tikz} % ??????
\usetikzlibrary{arrows.meta}
\usepackage{tikz}
\usepackage{amsmath}
\usetikzlibrary{arrows.meta, bending, calc}
\usepackage{amssymb,amsfonts,amsmath}
\usepackage{float}
\usepackage{booktabs}
 \usepackage{adjustbox}
\usepackage{array}
\usepackage{graphicx}
\usepackage{multirow}
\usepackage{color}
\def\blue{\textcolor{blue}}

\def\blue{\textcolor{blue}}
\usepackage{amsfonts, amsthm, amsmath}
\allowdisplaybreaks[4]
\usepackage{pgfplots}
\usepackage{xcolor}
\usepackage{afterpage}
\usepackage{rotating}

\usepackage{tikz}
\usepackage{lmodern}

\usetikzlibrary{decorations.pathmorphing}
\tikzset{
	redline/.style={red, thick, ->, shorten >=1pt, shorten <=1pt}
}

\usepackage{xcolor}
\usepackage{framed}
\usepackage{graphics}

\usepackage{amssymb}

\usepackage{amscd}

\usepackage{t1enc}

\usepackage[mathscr]{eucal}

\usepackage{indentfirst}

\usepackage{enumitem}

\usepackage{enumitem}

\usepackage{graphicx}

\usepackage{graphics}

\usepackage{pict2e}

\usepackage{mathrsfs}

\usepackage{comment}

\usepackage{enumitem}

\usepackage{enumerate}
\usepackage[pagebackref]{hyperref}
\hypersetup{colorlinks=true}%give color to the clickable link
\usepackage{color}
\usepackage{epic}
\usepackage{framed}
\usepackage{mathabx}
\usepackage{booktabs}
\usetikzlibrary{decorations.pathreplacing,calc}
\usepackage{makecell}
\newcolumntype{V}{!{\vrule width 2pt}}

\numberwithin{equation}{section}

\def\blue{\textcolor{blue}}

\def\blue{\textcolor{blue}}

\theoremstyle{plain}

\newtheorem{theorem}{Theorem}[section]

\newtheorem{conjecture}[theorem]{Conjecture}
\newtheorem{observation}[theorem]{Observation}

\newtheorem{lemma}[theorem]{Lemma}

\definecolor{handred}{RGB}{220, 0, 0}
\definecolor{handgreen}{RGB}{0, 150, 0}

\def\Rec{\mathsf{Rec}}
\def\des{\mathsf{des}}
\def\ear{\mathsf{ear}}
\def\cross{\mathsf{cross}}
\def\cyc{\mathsf{cyc}}
\def\fix{\mathsf{fix}}
\def\pex{\mathsf{pex}}
\def\pcyc{\mathsf{pcyc}}

\def\Ear{\mathsf{Ear}}

\def\exc{\mathsf{exc}}
\def\pdrop{\mathsf{pdrop}}

\def\Des{\mathsf{Des}}
\def\rmin{\mathsf{rmin}}
\def\nest{\mathsf{nest}}

\def\L{\mathcal{L}}

\def\Rmin{\mathsf{Rmin}}

\def\Pdrop{\mathsf{Pdrop}}

\begin{document}
	\begin{center}
		{\Large\bf On two conjectures concerning  special kinds of descents on permutations}
	\end{center}
	
	\begin{center}
		
		{\small Taifeng Ding, 
			Sherry H.F. Yan$^{*}$\footnote{$^*$Corresponding author. \\{\em E-mail address:}   hfy@zjnu.cn (S.H.F. Yan). }} 
		
		 Department of Mathematics,
		Zhejiang Normal University\\
		Jinhua 321004, P.R. China

	\end{center}

	\noindent {\bf Abstract.}
In this paper, we prove the continued fraction conjecture posed by Han, Mao and Zeng for the generating function of permutations with respect to the number of descents of type $2$ and the number of cycles, thereby settling their reformulation of a conjecture originally due to Baril and Kirgizov. We further establish the symmetry of the bistatistic $(\des_2, \ear)$ over $\mathfrak{S}_n$ as conjectured by Han, Mao and Zeng and strengthen this result by exhibiting  five equidistributed companions for $(\des_2, \ear)$. Here the statistic $\des_2$ denotes the number of descents of type $2$, and  the statistic $\ear$ denotes the number of exclusive antirecord cycle peaks originally introduced by Sokal and Zeng.


	\section{Introduction}
	Since MacMahon's pioneering work \cite{Mac}, permutation statistics have constituted a classical subject in enumerative combinatorics; see \cite{Huang-lin-yan jcta, Huang-lin-yan2, Huang-yan AAM, Liujcta, Liu1, Liu2} for recent developments.
	Let $\mathfrak{S}_n$ denote the set of permutations of $[n]=\{1,2,\ldots, n\}$.   Given a permutation $\pi=\pi_1\pi_2\ldots \pi_n$, an index $i$ ($1\leq i<n$) is said to be a {\em descent}
	if $\pi(i)>\pi(i+1)$ and an index $i$ ($1\leq i\leq n$) is said to be an {\em excedance} if $\pi(i)>i$.     Denote by $\des(\pi)$ (resp., $\exc(\pi)$) the number of descents (resp., excedences) of $\pi$, respectively.  
	It is a classical result in  enumerative combinatorics that the statistics $\des$ and $\exc$ are equally  distributed over $\mathfrak{S}_n$ (see~\cite[Chapter 1]{ST}) and their common  generating function is given by Eulerian polynomials $A_n(t)$, i.e.,  
	$$
	A_n(t)=\sum_{\pi \in \mathfrak{S}_n} t^{\des(\pi)}=\sum_{\pi \in \mathfrak{S}_n} t^{\exc(\pi)}.
	$$
	Following \cite[Chapter 1]{ST}, the {\em cycle notation}  of a permutation $\pi$ is obtained as follows. A sequence $(a_1, a_2, \ldots, a_k)$ is called
	a {\em cycle } of $\pi$ if $\pi(a_i)=a_{i+1}$ for all $1\leq i\leq k$ with the convention that 
	$a_{k+1}=a_1$.   Then   $\pi$ can be decomposed  into a disjoint union of cycles $C_1$, $C_2$, $\ldots, C_s$  and the {\em cycle notation} of $\pi$  is given by $\pi=C_1C_2\ldots C_s$. For example, the cycle notation  of $\pi=2461357$ is given by $\pi=(124)(365)(7)$.  In what follows,  each cycle is written with its smallest element first. A cycle of $\pi$ is said to be {\em pure } if it contains at least two elements.   Let $\cyc(\pi)$ and $\pcyc(\pi)$ denote  the number of cycles and the number of pure cycles in the cycle notation of $\pi$, respectively.
	
	 Given a permutation $\pi=\pi(1)\pi(2)\ldots \pi(n)\in \mathfrak{S}_n$, an index $i$ ($1\leq i\leq n$) is said to be 
	\begin{itemize}
		\item a {\em record (or left-to-right maximum)}  if either $i=1$ or $\pi(j)<\pi(i)$ for all $j<i$;
		\item an {\em antirecord (or right-to-left minimum)} if either $i=n$ or $\pi(j)>\pi(i)$  for all $j>i$. 
		\item a	{\em descent of type $2$} if  $i$ is both a descent  and  a record;
		\item a {\em pure excedence} if $i$ is an excedence and $\pi(j)\notin [i, \pi(i)]$ for all $j<i$, where $[a,b]=\{k\mid a\leq k\leq b\}$  for positive integers $a,b$;
		\item a {\em drop} if $\pi(i)<i$;
		\item a {\em pure drop} if $\pi(i)<i$ and $\pi(j)\notin [\pi(i),  i]$ for any $j>i$.
		\item a {\em fixed point} if $\pi(i)=i$.
		\item a {\em cycle peak} if $\pi(i)<i>\pi^{-1}(i)$;
		\item an {\em exclusive antirecord cycle peak }  if $i$ is a cycle peak and an antirecord but not a record. 
	\end{itemize}
	Let $\des_2(\pi)$ (resp., $\pex(\pi)$, $\pdrop(\pi)$, $\fix(\pi)$, $\ear(\pi)$) denote the number of descents of type $2$ (resp., pure excedances, pure drops, fixed points,  exclusive antirecord cycle peaks) of $\pi$.     For instance, if $\pi=6571423$,  then    $\des_2(\pi)= \pex(\pi)=\pdrop(\pi)=\ear(\pi)=2$ and $\fix(\pi)=0$.  The statistic $\ear$ was originally introduced by Sokal and Zeng \cite{sokal2022multivariate}.  In \cite{baril2021transformation}, Baril and Kirgizov introduced the statistics $\des_2$, $\pex$ and $\pcyc$, and proved bijectively that these statistics are equally distributed over $\mathfrak{S}_n$. They further concluded with two interesting 
	conjectures on the equidistribution of bistatistics. 
	
	\begin{conjecture}\label{con1}{\upshape(Baril and Kirgizov) }The two bistatistics $(\des_2, \cyc)$ and $(\pex, \cyc)$ are
		equidistributed on $\mathfrak{S}_n$. 
	\end{conjecture}
	
	\begin{conjecture}\label{con2}{\upshape(Vajnovszki) }The two bistatistics $(\des, \des_2)$ and $(\exc, \pex)$ are
		equidistributed on $\mathfrak{S}_n$. 
	\end{conjecture}
	It should be noted that Conjecture \ref{con2} has been resolved by Han, Mao and Zeng \cite{han2021equidistributions} through the
	combinatorial theory of $J$-continued fractions developed by Flajolet and Viennot in
	the 1980s \cite{flajolet1980combinatorial, francon1979permutations}.   In fact,  Han, Mao and Zeng proved that  the four bistatistics $(\exc, \pcyc)$, $(\exc, \ear)$, $(\des, \des_2)$ and $(\exc, \pex)$ are equidistributed over $\mathfrak{S}_n$ (see \cite[Corollary 1.6]{han2021equidistributions}).
	
	Define the polynomials $A_n(t,\lambda,y,w)$ by the $J$-fraction
	\begin{equation}\label{eq:1.2}
		\sum_{n\ge 0} A_n(t,\lambda,y,w)\, z^n
		=\cfrac{1}{
			1-\gamma_0z-\cfrac{\beta_1\,z^2}{
				1-\gamma_1z-\cfrac{\beta_2\,z^2}{
					1-\gamma_2z-\cdots
				}
			}
		},
	\end{equation}
	with
	\[
	\gamma_n = w + n(t+1),\qquad
	\beta_n = t(\lambda+n-1)(y+n-1).
	\]
	See \cite{sokal2022multivariate} for  generalizations of the polynomials $A_n(t,\lambda,y,w)$. In \cite{han2021equidistributions}, Han, Mao and Zeng investigated the combinatorial interpretations of the polynomials $A_n(t,\lambda,y,w)$ in terms of various permutation statistics.
	In particular, they   obtained five equidistributed companions of the bistatistic $(\pex,\cyc)$ in Conjecture \ref{con1} (see \cite[Theorem 1.7]{han2021equidistributions}), that is, 
	\begin{equation}\label{eq:han}
		\begin{aligned}
			A_n(1,\lambda,y,\lambda)
			&=\sum_{\sigma\in\mathfrak{S}_n} y^{\pex(\sigma)}\lambda^{\ear(\sigma)+\fix(\sigma)}
			=\sum_{\sigma\in\mathfrak{S}_n} y^{\ear(\sigma)}\lambda^{\pex(\sigma)+\fix(\sigma)} \\[4pt]
			&=\sum_{\sigma\in\mathfrak{S}_n} y^{\pcyc(\sigma)}\lambda^{\ear(\sigma)+\fix(\sigma)}
			=\sum_{\sigma\in\mathfrak{S}_n} y^{\ear(\sigma)}\lambda^{\cyc(\sigma)} \\[4pt]
			&=\sum_{\sigma\in\mathfrak{S}_n} y^{\pcyc(\sigma)}\lambda^{\pex(\sigma)+\fix(\sigma)}
			=\sum_{\sigma\in\mathfrak{S}_n} y^{\pex(\sigma)}\lambda^{\cyc(\sigma)}.  
		\end{aligned}
	\end{equation}

	By (\ref{eq:han}), Han, Mao and Zeng derived the following reformulation of Conjecture \ref{con1}.
	\begin{conjecture}{\upshape ( \cite{han2021equidistributions}, Conjecture 4.2)}
		\label{con:4.2}
		\[
		\sum_{n\geq 0}\sum\limits_{\sigma\in \mathfrak{S}_n}y^{\des_2(\sigma)}\lambda^{\cyc(\sigma)}z^{n}=\cfrac{1}{
			1-a_0x-\cfrac{b_1 z^2}{
				1-a_1x-\cfrac{b_2 z^2}{
					1-a_2x-\cfrac{b_3 z^2}{
						1-a_3x-\cdots
		}}}}
		\]
		with
		\[
		a_n = \lambda + 2n,\qquad
		b_n = (\lambda + n-1 )(y + n-1 ).
		\]
		
	\end{conjecture}

In \cite{han2021equidistributions}, Han, Mao, and Zeng observed that the distribution of the bistatistic $(\des_2, \pex)$ is not symmetric over $\mathfrak{S}_n$, and further conjectured that the distribution of the bistatistic $(\des_2, \ear)$ is symmetric over $\mathfrak{S}_n$.
	
	\begin{conjecture}{\upshape(\cite{han2021equidistributions}, Conjecture 4.1)}\label{con4.1}
		The distribution of $(\des_2, \ear)$ over permutations is symmetric, i.e., 
		$$
		\sum_{\pi\in \mathfrak{S}_n
		}q^{\des_2(\pi)}t^{\ear(\pi)}= \sum_{\pi\in \mathfrak{S}_n
		}q^{\ear(\pi)}t^{\des_2(\pi)}.
		$$
	\end{conjecture}
	The main objective of this paper is to  settle Conjectures \ref{con:4.2} and \ref{con4.1}.  In fact, we confirm and strengthen Conjecture   \ref{con4.1} by exhibiting  five  equidistributed companions of $(\des_2, \ear)$.

	\begin{theorem}\label{thm-equi}
		The six  bistatistics $(\des_2, \ear)$,  $(\ear, \des_2)$, $(\des_2, \pdrop)$, $(\pdrop, \des_2)$,  $(\des_2, \pcyc)$, $(\pcyc, \des_2)$ are
		equidistributed on $\mathfrak{S}_n$.    
	\end{theorem}
The rest of this paper is organized as follows. In Section 2, we prove Conjecture \ref{con:4.2} by transforming the bistatistic $(\des_2, \cyc)$ on permutations into certain bistatistic on weighted Motzkin paths. In Section 3, we construct an involution on $\mathfrak{S}_n$ that establishes the symmetry of the bistatistic $(\des_2, \ear)$. In Section 4, we strengthen Conjecture \ref{con4.1} by establishing the equidistribution of the bistatistics $(\des_2, \ear)$, $(\des_2, \pdrop)$, $(\des_2, \pcyc)$ over $\mathfrak{S}_n$.
    \section{Proof of Conjecture \ref{con:4.2}}
    This section is devoted to the proof of Conjecture \ref{con:4.2}. This is accomplished by transforming the bistatistic $(\des_2, \cyc)$ on permutations into certain bistatistic  on weighted Motzkin paths.

    Let $\mathcal{I}_n$ denote the set of sequences $(a_1, a_2, \ldots, a_n)$ with $1\leq a_i\leq i$ for all $1\leq i\leq n$.  There is a natrual  bijection $\theta$ between $\mathfrak{S}_n$ and $\mathcal{I}_n$ (see \cite[Chapter 1]{ST}). Given a permutation $\pi=\pi(1)\pi(2)\cdots\pi(n)\in \mathfrak{S}_n$, we obtain a sequence $\theta(\pi)=(a_1,a_2,\ldots,a_n)$, where for each $i$,
    \[
    a_i = \#\,\{\,j \mid j\leq i \text{ and } \pi^{-1}(j)\leq \pi^{-1}(i)\,\}.
    \]
     For example, if $\pi=3527164$, then $\theta(\pi)=(1,1, 1,4,2,5,4)$.
   
   Conversely, given an inversion sequence $s=(a_1, a_2, \ldots, a_n)$, we can obtain a permutation $\theta^{-1}(s)=\pi=\pi(1)\pi(2)\cdots\pi(n)\in \mathfrak{S}_n$ recursively as follows.
   Let $\pi^{(1)}=1$. Assume that $\pi^{(i-1)}$ has been determined. Generate a permutation $\pi^{(i)}$ from $\pi^{(i-1)}$ by inserting the letter $i$ immediately before the $a_i$-th letter of $\pi^{(i-1)}$ if $a_i < i$, and putting $i$ at the end of $\pi^{(i-1)}$ otherwise. Define $\theta^{-1}(s)=\pi^{(n)}$.

     In the following, we construct another bijection $\chi$ between $\mathcal{I}_n$ and $\mathfrak{S}_n$. Given a sequence $s=(a_1,a_2,\dots,a_n)\in \mathcal{I}_n$, we generate a sequence of permutations $\sigma^{(1)}, \sigma^{(2)}, \ldots,\sigma^{(n)}$ as follows. Set $\sigma^{(1)}=(1)$. Assume that $\sigma^{(i-1)}$ has already been determined. If $a_i = i$, define $\sigma^{(i)}$ to be the permutation whose cycle notation is obtained from that of $\sigma^{(i-1)}$ by inserting the cycle $(i)$. Otherwise,  define $\sigma^{(i)}$ to be the permutation whose cycle notation is  obtained from that of $\sigma^{(i-1)}$ by inserting the entry $i$ immediately after the entry $a_i$. Set  $\chi(s)=\sigma^{(n)}$. For example, if $s=(1,1,1,4,2,5, 4)$, then\begin{align*}
     	\chi(s)&=   (1)\rightarrow (12)\rightarrow (132) \rightarrow (132)(4)\rightarrow (1325)(4)\\
     	&\quad\rightarrow  (13256)(4) \rightarrow (13256)(47)\rightarrow 3527614
     \end{align*}
     
     \begin{lemma}\label{lem-chi}
     	For $n\geq 1$, the map $\chi$ is a bijection between $\mathcal{I}_n$ and $\mathfrak{S}_n$.
     	\end{lemma}
     	\begin{proof}
     In order to show that the map $\chi$ is a bijection, we first give a description of a map $\chi'\colon \mathfrak{S}_n\rightarrow \mathcal{I}_n$. Given a permutation
     $\sigma\in \mathfrak{S}_n$, we recursively generate a sequence of pairs $(\sigma^{(n)}, a_n)$, $(\sigma^{(n-1)}, a_{n-1})$, $\ldots, (\sigma^{(1)}, a_1)$,  where each pair
     $(\sigma^{(i)}, a_i)$ is determined by the following rules.
     \begin{itemize}
     	\item $\sigma^{(n)}=\sigma$. 
     	\item For all $1\leq i<n$, the cycle notation of $\sigma^{(i)}$ is obtained from that of $\sigma^{(i+1)}$ by removing the letter $i+1$.
     	\item If the cycle notation of $\sigma^{(i)}$ contains the cycle $(i)$, let $a_i=i$. Otherwise,
     	let $a_i$ be the letter immediately before the letter $i$ in the cycle notation of $\sigma^{(i)}$.
     \end{itemize}
     Set $\chi'(\sigma)=(a_1, a_2, \ldots, a_n)$. Clearly, we have $\chi'(\sigma)\in \mathcal{I}_n$, and hence the map $\chi'$ is well-defined. Moreover, the maps $\chi$ and $\chi'$ are inverses of each other, and hence  the map $\chi$ is a bijection. This completes the proof.
     \end{proof}
     Let $\Des_2(\pi)$ denote the set of descents of type  $2$ of $\pi$ and let $\rmin(\pi)$ denote the number of right-to-left minima of $\pi$.  
     \begin{theorem}\label{thm-Des_2}
     The map $\delta:=\chi\circ\theta$ induces a bijection from $\mathfrak{S}_n$ onto itself such that for any permutation $\pi\in \mathfrak{S}_n$, we have
     	\begin{equation}\label{eq-Des_2}
     	(\Des_2, \rmin)\pi= (\Des_2, \cyc)\delta(\pi).
     		\end{equation}
     	\end{theorem}
     \begin{proof}
    Since both   $\theta$ and $\chi$ are bijections, the map $\delta$   is a bijection from   $\mathfrak{S}_n$ onto itself.  Now we proceed to  show that \eqref{eq-Des_2} holds for any $\pi\in \mathfrak{S}_n$. Let $\pi=\pi(1)\pi(2)\cdots\pi(n)$, $\theta(\pi)=s=(a_1,a_2,\ldots,a_n)$ and $\chi(s)=\sigma=\sigma(1)\sigma(2)\cdots\sigma(n)$. When we apply the map $\theta^{-1}$ to $s=(a_1,a_2,\ldots,a_n)$, we generate a sequence of permutations $\pi^{(1)},\pi^{(2)},\ldots,\pi^{(n)}$ such that $\pi^{(n)}=\pi$. Recall that $\pi^{(1)}=1$, and for all $i>1$, the permutation $\pi^{(i)}$ is obtained from $\pi^{(i-1)}$ by inserting the letter $i$ immediately before the $a_i$-th letter of $\pi^{(i-1)}$ if $a_i<i$, and putting the letter $i$ at the end of $\pi^{(i-1)}$ otherwise.
     
     Similarly, when applying the map $\chi$ to $s$, we also obtain a sequence of permutations $\sigma^{(1)},\ldots,\sigma^{(n)}$, where $\sigma^{(1)}=1$ and $\sigma^{(n)}=\sigma$. For  each  $i>1$,   $\sigma^{(i)}$ is determined by the following rules.
     \begin{itemize}
     	\item If $a_i = i$, define   $\sigma^{(i)}$ to be the permutation  whose cycle notaion is obtained from that  of $\sigma^{(i-1)}$ by inserting the cycle $(i)$.
     	\item Otherwise, define   $\sigma^{(i)}$ to be the permutation  whose cycle notaion is obtained from that of   $\sigma^{(i-1)}$ by  inserting  the entry $i$ immediately after the entry $a_i$. 
     \end{itemize}
     
     In what follows, we aim to show that
     \[
     (\Des_2, \rmin)\pi^{(i)}= (\Des_2, \cyc)\sigma^{(i)}
     \]
     holds for all $1\leq i\leq n$. We proceed by induction on $i$. Clearly, the assertion holds for $i=1$ since $\pi^{(1)}=\sigma^{(1)}=1$. Now assume that $i>1$, and suppose that
     \[
     (\Des_2, \rmin)\pi^{(i-1)}= (\Des_2, \cyc)\sigma^{(i-1)}.
     \]
     
     By the construction of $\pi^{(i)}$, it is easily seen that
     \[
     \Des_2(\pi^{(i)})= \{k\mid  k\in \Des_2(\pi^{(i-1)}) \;\; \mbox{and}\;\; k<a_i\}  \cup\{a_i\},\quad \rmin(\pi^{(i)})=\rmin(\pi^{(i-1)})
     \]
     when $a_i<i$, and
     \[
     \Des_2(\pi^{(i)})= \Des_2(\pi^{(i-1)}),\quad \rmin(\pi^{(i)})=\rmin(\pi^{(i-1)})+1
     \]
     otherwise.
     
     Similarly, by the construction of $\sigma^{(i)}$, one can easily check that if $a_i=i$, then $\sigma^{(i)}(j)=\sigma^{(i-1)}(j)$ for all $j<i$ and $\sigma^{(i)}(i)=i$. Moreover, if $a_i<i$, then $\sigma^{(i)}(j)=\sigma^{(i-1)}(j)$ for all $j<i$ and $j\neq a_i$, $\sigma^{(i)}(a_i)=i$ and $\sigma^{(i)}(i)=\sigma^{(i-1)}(a_i)$. Therefore, we deduce that
     \[
     \Des_2(\sigma^{(i)})= \{k\mid  k\in \Des_2(\sigma^{(i-1)})\;\; \mbox{and}\;\; k<a_i \}\cup\{a_i\},\quad \cyc(\sigma^{(i)})=\cyc(\sigma^{(i-1)})
     \]
     when $a_i<i$, and
     \[
     \Des_2(\sigma^{(i)})= \Des_2(\sigma^{(i-1)}),\quad \cyc(\sigma^{(i)})=\cyc(\sigma^{(i-1)})+1
     \]
     otherwise.
     
     Then, by the induction hypothesis, we conclude that
     \[
     \Des_2(\pi^{(i)})=\Des_2(\sigma^{(i)})\quad \mbox{and}\quad \rmin(\pi^{(i)})=\cyc(\sigma^{(i)})
     \]
     for all $1\leq i\leq n$. Hence, we deduce that
     \[
     (\Des_2, \rmin)\pi=  (\Des_2, \rmin)\pi^{(n)}= (\Des_2, \cyc)\sigma^{(n)}=(\Des_2, \cyc)\sigma
     \]
     as desired, completing the proof. 
     	\end{proof}

      Recall that a  {\em  Motzkin path} of length $n$  is a lattice path in the quarter plane $\mathbb{N}^2$ starting from  $(0,0)$ to $(n,0)$   using three  possible steps:
      $$
      \text{$(1,1)=U$ (up step), $(1,-1)=D$ (down step)\,\,\,and\,\,\,$(1,0)=H$ (horizontal step)  } 
      $$
      and never lying below the $x$-axis. 
      A $2$-Motzkin path is a Motzkin path in which the horizontal steps can be colored by red and  blue. For our convenience, we use $H_1$ (resp., $H_2$) to denote the red (resp., blue) horizontal step.  Let $\mathcal{M}^2_n$  denote the set of  $2$-Motzkin paths of length $n$. 
      For a   Motzkin path $P=p_1p_2\cdots p_n$ with $n$ steps, let $h_i(P)$ be the {\em height} of the $i$-th step of $P$:
      $$
      h_i(P):=|\{j\mid j<i, p_j=U\}|-|\{j\mid j<i, p_j=D\}|.
      $$
      A {\em restricted Laguerre history} of length $n$ is a pair $(M,w)$, where $M=m_1m_2\cdots m_n\in\mathcal{M}^2_n$ and $w=w_1w_2\cdots w_n\in \mathbb{N}^n$ is a weight function satisfying
      $$
      0\leq w_i\leq \begin{cases}
      	h_i(M),\quad&\text{if $m_i=U, H_1$};\\
      	h_i(M)-1, \quad&\text{if $m_i=D, H_2$}.
      \end{cases}
      $$
      Let $\mathcal{L}_n$ denote the set of all restricted Laguerre histories of length $n$. In the following, we   recall
      the classical Foata--Zeilberger bijection $\Phi_{FZ}$~\cite{FZ}, which is a variant of the Fran\c con- Viennot bijection \cite{francon1979permutations}. 
      Given a permutation $\pi\in \mathfrak{S}_n$, we use the convention
      $\pi(0)=0$ and $\pi(n+1)=+\infty$. For each $i\in[n]$,  define 
      $$(\underline{31}2)_i(\pi):=|\{k: k<j\text{ and } \pi(k)<\pi(j)=i<\pi(k-1)\}|.$$
      Define $\Phi_{FZ}(\pi)=(M,w)$, where for  $i\in[n]$ with $\pi(j)=i$:
      $$
      m_i=\left\{
      \begin{array}{ll}
      	U&\quad\mbox{if $\pi(j-1)>\pi(j)<\pi(j+1)$},  \\
      	D&\quad\mbox{if $\pi(j-1)<\pi(j)>\pi(j+1)$},  \\
      	H_1&\quad\mbox{if $\pi(j-1)<\pi(j)<\pi(j+1)$},\\
      	H_2&\quad\mbox{if $\pi(j-1)>\pi(j)>\pi(j+1)$},
      \end{array}
      \right.
      $$
      and $w_i=(\underline{31}2)_i(\pi)$.
       For example,  if $\pi=6715423\in\mathfrak{S}_7$, then
       $$\Phi_{FZ}(\pi)=(UUH_1H_2DH_1D,  0121100).$$
      
      The inverse algorithm $\Phi_{FZ}^{-1}$ constructing a permutation $\pi$ (in $n$ steps) from a Laguerre history $(M,w)\in\L_{n}$ can be described iteratively as:
      \begin{itemize}
      	\item Initialization: $\pi=\diamond$;
      	\item At the $i$-th ($1\leq i\leq n$) step of the algorithm, replace the $(w_i+1)$-th $\diamond$ (from left to right) of $\pi$ by
      	$$
      	\begin{cases}
      		\,\diamond i\diamond&\quad \text{if $m_i=U$},\\
      		\, i\diamond&\quad \text{if $m_i=H_1$},\\
      		\, i& \quad\text{if $m_i=D$},\\
      		\, \diamond i&\quad \text{if $m_i=H_2$};
      	\end{cases}
      	$$
      	\item The final permutation is obtained by removing  the last remaining $\diamond$.
      \end{itemize}
      
      For example,  if $(M, w)=(UUH_1H_2DH_1D,  0121100)$, then \begin{align*}
      	\pi&=\diamond\rightarrow \diamond1\diamond\rightarrow\diamond1\diamond2\diamond\rightarrow\diamond1\diamond23\diamond \rightarrow  \diamond1\diamond423\diamond\\
      	&\quad\rightarrow \diamond15423\diamond\rightarrow 6\diamond15423\diamond\rightarrow
  6715423\diamond\rightarrow 6715423.
      \end{align*}
      
      Given a permutation $\pi\in \mathfrak{S}_n$, define $$\Rmin(\pi)=\{\pi(k)\mid  \mbox{either $k=n$ or}\,\,  \pi(j)>\pi(k) \,\, \mbox{for all $j>k$}\}.$$ 
       
      The following property  follows directly from the definition of   $\Phi^{-1}_{FZ}$. 
      \begin{lemma}\label{lem-PhiFZ}
      Given a  Laguerre history $(M,w)\in\L_{n}$ with $M=m_1m_2\ldots m_n$ and $w=w_1w_2\ldots w_n$, let  $\pi=\Phi^{-1}_{FZ}(M,w)$. Then  we have
      	\begin{equation}\label{eq-des_2-Mot}
      	 \des_2(\pi)=|\{k\mid (m_k, w_k)=(D, 0)\}| 
      	 \end{equation}
      	and 
      	\begin{equation}\label{eq-Rec-Mot}
     \Rmin(\pi)=
      	\{k\mid (m_k, w_k)=(H_1, h_k(M))\}\cup \{k\mid (m_k, w_k)=(U, h_k(M))\}.
      	\end{equation}
      	\end{lemma}
      	 
      	 Given a restricted 
      	   Laguerre history $(M,w)\in\L_{n}$ with $M=m_1m_2\ldots m_n$ and $w=w_1w_2\ldots w_n$,  define
      	   $$
      	\alpha(M,w)=\#\,\{k\mid (m_k, w_k)=(D, 0)\},
      	   $$
      	   and 
      	  $$
      	  \beta(M,w)=  \#\,\{k\mid (m_k, w_k)=(U, h_k(M))\}+\#\,\{k\mid (m_k, w_k)=(H_1, h_k(M))\}|.
      	   $$

      	   The following property of $\Phi^{-1}_{FZ}$ follows directly from Lemma \ref{lem-PhiFZ}.
      	   \begin{theorem}\label{thm-PhiFZ}
      	The map $\Phi^{-1}_{FZ}$ is a bijection from $\mathfrak{S}_n$ onto itself	 such that for  any  Laguerre history $(M,w)\in\L_{n}$,  we have
      	$\pi=\Phi^{-1}_{FZ}(M,w)$ satisfying 
      		\begin{equation}\label{eq-des_2-Mot}
      			\des_2(\pi)= \alpha(M,w)\,\, \, \mbox{and}\,\,\, \rmin(\pi)=\beta(M,w).
      		\end{equation}
      			\end{theorem}
      
       Given a nonnegative integer $k$, a  restricted Laguerre history    $(M,w)$ with $M=m_1m_2\cdots m_n\in\mathcal{M}^2_n$ and $w=w_1w_2\cdots w_n\in \mathbb{N}^n$ is said to be of {\em rank $k$ }  if it satisfies  
      $$
      0\leq w_i\leq \begin{cases}
      	h_i(M)+k,\quad&\text{if $m_i=U, H_1$};\\
      	h_i(M)+k-1, \quad&\text{if $m_i=D, H_2$}.
      \end{cases}
      $$
      Clearly,   An ordinary   restricted Laguerre history    $(M,w)$   has rank $0$.    Let $\mathcal{L}_n^k$ denote the set of restricted Laguerre histories  $(M,w)$ of rank $k$, where $M\in \mathcal{M}^2_n$.

     \begin{lemma}\label{lem-gen}
     	\begin{equation}\label{eq-con}
     	\sum_{n\geq 0}\sum_{(M, w)\in \mathcal{L}_n} z^ny^{\alpha(M,w)}\lambda^{\beta(M,w)}=\cfrac{1}{
     		1-a_0x-\cfrac{b_1z^2}{
     			1-a_1x-\cfrac{b_2 z^2}{
     				1-a_2x-\cfrac{b_3 z^2}{
     					1-a_3x-\cdots
     	}}}}
     	\end{equation}
     	with
     	\[
     a_n = \lambda + 2n,\qquad
    b_n = (\lambda + n-1 )(y + n -1).
     	\]
     	\end{lemma}
     	\begin{proof}
     		Given a  nonempty restricted Laguerre history    $(M,w)$ of rank $k$,  the path $M$ can be uniquely decomposed as   one of the
     		following three forms:
     		\[
     		\text{(i)}\quad  M=H_1M'
     		\qquad
     		\text{(ii)}\quad M=H_2M'
     		\qquad
     		\text{(iii)}\quad M=UM''DM',
     		\]
     		where $M'$ and $M''$  are  (possibly empty)  $2$-Motzkin paths.   
     		Given a restricted 
     		Laguerre history $(M,w)\in\L^k_{n}$ with $M=m_1m_2\ldots m_n$ and $w=w_1w_2\ldots w_n$,  define
     		$$
     		\beta^k(M,w)=  \#\,\{k\mid (m_k, w_k)=(U, h_k(M)+k)\}+\#\,\{k\mid (m_k, w_k)=(H_1, h_k(M)+k)\}|.
     		$$
     		Clearly, we have $\beta^{0}(M,w)=\beta(M,w)$.
     		
     		Define
     		$$
     		F_k(z;y,\lambda)=\sum_{n\geq 0}\sum_{(M, w)\in \mathcal{L}^k_n} z^ny^{\alpha(M,w)}\lambda^{\beta^k(M,w)}.
     		$$
     	From the decomposition of $M$, it follows that 
     	$$
     	F_k(z;y, \lambda)=1+(\lambda+2k)zF_{k}(z;y,\lambda)+(\lambda+k)(y+k)F_{k+1}(z;y,\lambda)F_{k}(z;y, \lambda)
     	$$
     	for all $k\geq 0$. 
     	This yields
     	\[
     	F_k(z;y,\lambda)=\frac{1}{1-(\lambda+2k)z-(\lambda+k)(y+k)F_{k+1}(z;y,\lambda)},
     	\]
     	which implies (\ref{eq-con}), completing the proof. 
     		\end{proof}
     		
     		\noindent{\bf Proof of Conjecture \ref{con:4.2}.}  In view of Theorems \ref{thm-Des_2} and  \ref{thm-PhiFZ} and Lemma \ref{lem-gen}, we deduce that
     			\[
     			\begin{array}{lll}
     		\sum\limits_{n\geq 0}\sum\limits_{\sigma\in \mathfrak{S}_n}y^{\des_2(\sigma)}\lambda^{\cyc(\sigma)}z^{n}&=&\sum\limits_{n\geq 0}\sum\limits_{\sigma\in \mathfrak{S}_n}y^{\des_2(\sigma)}\lambda^{\rmin(\sigma)}z^{n}\\
     		&=&	\sum\limits_{n\geq 0}\sum\limits_{(M, w)\in \mathcal{L}_n} z^ny^{\alpha(M,w)}\lambda^{\beta(M,w)}\\
     		&=&\cfrac{1}{
     			1-a_0x-\cfrac{b_1 z^2}{
     				1-a_1x-\cfrac{b_2z^2}{
     					1-a_2x-\cfrac{b_3 z^2}{
     						1-a_3x-\cdots
     		}}}}
     	\end{array}
     		\]
     		with
     		\[
     		a_n = \lambda + 2n,\qquad
     		b_n = (\lambda + n -1)(y + n -1).
     		\]
     		This completes the proof. \qed
     \section{Proof of Conjecture \ref{con4.1}} 
    In this section, we aim to construct an involution on $\mathfrak{S}_n$ that transforms the bistatistic $(\des_2, \ear)$ into $(\ear, \des_2)$, thereby confirming Conjecture \ref{con4.1}.
    To this end, we describe another bijection $\Psi_{FZ}$ of Foata and Zeilberger \cite{FZ} between permutations and restricted Laguerre histories.
    Here we adopt the equivalent definition of $\Psi_{FZ}$ given by Corteel \cite{Corteel2007}.

      Given a permutation $\pi\in \mathfrak{S}_n$, 
       define $\Psi_{FZ}(\pi)=(M,w)$, where for each   $i\in[n]$
     $$
     m_i=\left\{
     \begin{array}{ll}
     	U&\quad\mbox{if $ \pi(i)>i<\pi^{-1}(i)$},  \\
     	D&\quad\mbox{if $ \pi(i)<i>\pi^{-1}(i)$},  \\
     	H_1&\quad\mbox{if $ \pi(i)\geq i\geq \pi^{-1}(i)$},\\
     	H_2&\quad\mbox{if $ \pi(i)<i<\pi^{-1}(i)$},
     \end{array}
     \right.
     $$
     and $$
     w_i=\nest_i(\pi):=	\begin{cases}
     	\#\,\{j\mid j<i\leq \pi(i)<\pi(j)\}&\quad \text{if $\pi(i)\geq i$},\\
     		\#\,\{j\mid \pi(j)<\pi(i)<i<j \}&\quad \text{if $\pi(i)<i$}.\\ 
     \end{cases}
     $$

     For example, if $\pi=641958(10)732$, then we have
     $$\Psi_{FZ}(\pi)=(UUH_2H_1H_1H_1UDDD, 0100210210).$$
     
    The inverse algorithm   $\Psi_{FZ}^{-1}$ generating a permutation $\pi$  from a Laguerre history $(M,w)\in\L_{n}$  can be discribed as follows.
     \begin{itemize}
     	
     	\item Let $F=\{i\mid m_i=U\,\, \mbox{or}\,\, H_1\}$, $F'=\{i\mid m_i=D \,\, \mbox{or}\,\, H_1\}$, $G=\{i\mid m_i=D\,\, \mbox{or}\,\, H_2\}$ and $G'=\{i\mid m_i=U\,\, \mbox{or}\,\, H_2\}$. 
     	\item Let $f$ and $g$ be the increasing permutations of $F$ and $G$, respectively.
     	
     	\item Construct the biword $\binom{f}{f'}$: for each $i$ in the first row $f$, processed from 
     	the \blue{ largest}  (rightmost) element , the entry in $f'$ that is below  $i$ is the $(w_i+1)$-th \blue{largest } element $j\in F'$ such that $j\geq i$ and $j$ has not yet been chosen.
     	
     	\item Construct the biword $\binom{g}{g'}$: for each $i$ in the first row $g$, processed from the \blue{smallest } (leftmost) element , the entry in $g'$ that is below  $i$ is the $(w_i+1)$-th \blue{smallest} element $j\in G'$ such that $j< i$ and $j$ has not yet been chosen. 
     	
     	\item Rearrange the columns so that the top row is in increasing order, and we obtain the
     	permutation $\pi=\Psi^{-1}_{FZ}(M,w)$ as the bottom row of the rearranged biword.
     \end{itemize}
     
     For example,  if  $(M,w)=(UUH_2H_1H_1H_1UDDD, 0100210210)$, then we have
     $$
     F=\{1,2,4,5,6,7\}, F'=\{4,5,6,8,9,10\}, G=\{3,8,9,10\}\,\, \mbox{and}\,\,  G'=\{1,2,3,7\},
     $$
     
     \[
     \binom{f}{f'}=
     \begin{pmatrix}
     	1 & 2 &  4 & 5 & 6 & 7 \\
     	6 & 4 & 9 & 5 & 8 & 10
     \end{pmatrix}, 
     \]
     and 
     \[
     \binom{g}{g'}=
     \begin{pmatrix}
     	3 &  8 & 9 & 10 \\
     	1 &   7 & 3 & 2
     \end{pmatrix}
     \]
     Rearrange the columns, we obtain
     \[ \begin{pmatrix}
     	1&2&3&4&5&6&7&8&9&10 \\
     	6&4&1&9&5&8&10&7&3&2
     \end{pmatrix},\]
     and hence $\pi=\Psi^{-1}_{FZ}(M,w)=641958(10)732$.
     
    Following \cite{Corteel2007},  given a permutation $\pi$, define
     $$
     \cross_i(\pi):=	\begin{cases}
     	\#\,\{j\mid j<i\leq \pi(j)<\pi(i)\}&\quad \text{if $\pi(i)\geq i$},\\
     	\#\,\{i\mid \pi(i)<\pi(j)< i<j \}&\quad \text{if $\pi(i)<i$}.\\ 
     \end{cases}
     $$
     
     \begin{lemma}\label{lem-Cor}  \upshape{(\cite{Corteel2007}, Lemma 3)}
     		Given a permutation $\pi\in \mathfrak{S}_n$,  let  $(M,w)=\Psi_{FZ}(\pi)$,  where  $M=m_1m_2\ldots m_n$ and $w=w_1w_2\ldots w_n$. Then  for each $i$, 
     		we have  
     	$$
     	\nest_i(\pi)+\cross_i(\pi)=h_i(M)
     	$$
     	when $\pi(i)\geq i$, 
     	and  
     		$$
     	\nest_i(\pi)+\cross_i(\pi)=h_i(M)-1
     	$$
     	otherwise.
     	\end{lemma}
     Let $\Rec(\pi)$  (resp., $\Ear(\pi)$, $\Pdrop(\pi)$) denote the set of records  (resp., exclusive antirecord cycle peaks, pure drops) of $\pi$. 
      \begin{observation}\label{ob0}
     	For any permutation $\pi\in \mathfrak{S}_n$, we have 
     	$$
     	\Rec(\pi)=\{i\mid \pi(i)\geq i,\, \nest_i(\pi)=0\},
     	$$
     	$$
     	\Pdrop(\pi)=\{i\mid \pi(i)<i>\pi^{-1}(i), \,\, \cross_i(\pi)=0\},
     	$$
     	$$
     	\Ear(\pi)=\{i\mid \pi(i)<i>\pi^{-1}(i), \,\nest_i(\pi)=0\},
     	$$
     	$$
     	\Rec(\pi^{-1})=\Rmin(\pi),
     	$$
     	and 
     	$$
     	\Des_2(\pi)=\{i\mid i\in \Rec(\pi), i+1\notin \Rec(\pi), 1\leq i<n\}. 
     	$$
     \end{observation}

     The following lemma follows readily from Lemma \ref{lem-Cor}, Observation \ref{ob0}, and the definition of $\Psi_{FZ}$.
     \begin{lemma}\label{lem-PsiFZ}
     	Given a  permutation $\pi\in \mathfrak{S}_n$,  let  $(M,w)=\Psi_{FZ}(\pi)$,  where  $M=m_1m_2\ldots m_n$ and $w=w_1w_2\ldots w_n$. Then  we have
     	\begin{equation}\label{eq-ear-Mot}
     		\ear(\pi)=\#\,\{k\mid (m_k, w_k)=(D, 0)\}
     	\end{equation}
     	\begin{equation}\label{eq-pdrop-Mot}
     		\pdrop(\pi)=\#\,\{k\mid (m_k, w_k)=(D, h_k(M)-1)\}
     	\end{equation}
     	
     	and 
     	\begin{equation}\label{eq-Rec-Mot}
     		\Rec(\pi)=\{k\mid (m_k, w_k)=(H_1, 0)\}| \cup\{k\mid (m_k, w_k)=(U, 0)\}.
     	\end{equation}
     \end{lemma}

     Given a  Laguerre history  $(M,w)\in \mathcal{L}_n$ with $M=m_1m_2\ldots m_n$ and $w=w_1w_2\ldots w_n$, we can constrcut a Laguerre history  $\kappa(M,w)=(M,w')\in \mathcal{L}_n$  with $w'=w'_1w'_2\ldots w'_n$ by letting
     $$
    w'_i=	\begin{cases}
      h_i(M)-w_i&\quad \text{if $m_i=U, H_1$ },\\
     w_i&\quad \text{otherwise}.\\ 
     \end{cases}
     $$
    Clearly, the map $\kappa$ is an involution on $\mathcal{L}_n$.

    In the following, 
   we construct a bijection $\Phi$ from  $\mathfrak{S}_n$ onto  itself  which transforms the concerned permutation statistics as follows. 
    \begin{theorem}\label{thm-Phi}
    	For $n\geq 1$, the map $\Phi:=\Phi^{-1}_{FZ}\circ \kappa \circ \Psi_{FZ}$  induces a bijection  from  $\mathfrak{S}_n$ onto itself such that for any permutation $\pi\in \mathfrak{S}_n$, we have $\sigma=\Phi(\pi)$ verifying that
    	$$
    	\ear(\pi)=\des_2(\Phi(\pi))\,\,\,  \mbox{and}\,\,\,  \des_2(\sigma^{-1})=\des_2(\pi).
    	$$
    	 
    	\end{theorem}
    	\begin{proof}	 
    Let $\pi\in \mathfrak{S}_n$, $(M,w)=\Psi_{FZ}(\pi)$, $(M,w')=\kappa(M,w)$ and $\sigma=\Phi^{-1}_{FZ}(M, w')$, where $M=m_1m_2\ldots m_n$, $w=w_1w_2\ldots w_n$ and $w'=w'_1w'_2\ldots w'_n$. 
By Lemmas \ref{lem-PhiFZ} and \ref{lem-PsiFZ} and the definition of $\kappa$, we deduce that 
   $$
   \ear(\pi)=\#\,\{k\mid (m_k, w_k)=(D, 0)\}=\#\,\{k\mid (m_k, w'_k)=(D, 0)\}=\des_2(\sigma)
   $$
   and 
   $$
   \begin{array}{lll}
   	 \Rec(\pi)&=&  \{k\mid (m_k, w_k)=(H_1, 0)\}| \cup\{k\mid (m_k, w_k)=(U, 0)\} \\
   	 &=&  \{k\mid (m_k, w'_k)=(H_1, h_k(M))\}| \cup\{k\mid (m_k, w'_k)=(U, h_k(M))\}\\
   	 &=& \Rmin(\sigma)\\
   	 &=& \Rec(\sigma^{-1})
   	\end{array}
   $$
   where the last equality follows from Observation \ref{ob0}.  Again by Observation \ref{ob0},  $\Rec(\pi)=\Rec(\sigma^{-1})$ would imply that 
  $\des_2(\sigma^{-1})=\des_2(\pi)$, completing the proof. 
  \end{proof}
  Let $\iota$ denote the map $\iota:\pi\rightarrow \pi^{-1}$. Now we are in position to construct the involution $\Psi$ on $\mathfrak{S}_n$ which transforms the bistatistic $(\ear, \des_2)$ into the bistatistic  $(\des_2, \ear)$. 
   \begin{theorem}\label{thm-Psi}
  	For $n\geq 1$, the map $\Psi:=\Phi^{-1}\circ \iota \circ \Phi$ is  an involution on    $\mathfrak{S}_n$   such that for any permutation $\pi\in \mathfrak{S}_n$, we have  
  	$$
  	(\ear, \des_2)\pi= (\des_2, \ear)\Psi(\pi).
  	$$
  \end{theorem}
    \begin{proof}
    In view of  Theorem \ref{thm-Phi}, the map $\Psi:=\Phi^{-1}\circ \iota \circ \Phi$ is indeed an involution on $\mathfrak{S}_n$. Given a permutation $\pi\in \mathfrak{S}_n$,  let   $\sigma=\Phi(\pi)$. 
    Again by Theorem \ref{thm-Phi}, we have
    \begin{equation}\label{eq-Psi}
     \ear(\Psi(\pi))=	\ear(\Phi^{-1}(\sigma^{-1}))=\des_2(\sigma^{-1})=\des_2(\pi).
    	\end{equation}
  This yields that
   \[
   \des_2(\Psi(\pi))=\ear\bigl(\Psi(\Psi(\pi))\bigr)=\ear(\pi),
   \]
  where the second equality follows from the fact that the map $\Psi$ is an involution.  Hence, we have conculded that $$
  	(\ear, \des_2)\pi= (\des_2, \ear)\Psi(\pi).
  $$
  as desired, completing the proof. 
    	 	\end{proof}

    	 \section{Proof of Theorem \ref{thm-equi}}
    	 This section is devoted to the proof of Theorem \ref{thm-equi}. 
    	First we construct a map $\Gamma:\mathcal{L}_n\rightarrow \mathfrak{S}_n$. 
    	To this end, we need the graphical representation of  a permutation. Following  \cite{Corteel2007}, 
    	for each permutation $\pi\in\mathfrak{S}_n$, we can associate to  $\pi$ a graphical representation by placing vertices $1,2,\dots,n$ along a horizontal axis  and then drawing an arc  connecting  $i$ to $\pi(i)$ above (resp.,  below) the horizontal axis in case $\pi(i)\geq i$ (resp., $\pi(i)<i$).  See Figure \ref{fig:graphic} for an example. In this geometrical setting,  each vertex thus has  $\text{out-degree}=\text{in-degree}=1$.
    		\begin{figure}[htbp]
    		\centering
    		\begin{tikzpicture}[x=0.9cm,y=0.9cm]

    			% permutation:
    			% pi = (6,4,1,9,5,8,10,3,7,2)
    			
    			\foreach \i in {1,...,10}{
    				\coordinate (v\i) at (\i,0);
    			}
    			
    			% baseline
    			\draw[line width=0.55pt] (0.55,0) -- (10.45,0);

    			\draw[line width=0.8pt,line cap=round,line join=round]
    			(v1) .. controls (2.05,1.18) and (4.95,1.18) .. (v6);
    			
    			\draw[line width=0.8pt,line cap=round,line join=round]
    			(v2) .. controls (2.45,0.54) and (3.55,0.54) .. (v4);
    			
    			\draw[line width=0.8pt,line cap=round,line join=round]
    			(v4) .. controls (5.05,1.48) and (7.95,1.48) .. (v9);
    			
    			\draw[line width=0.8pt,line cap=round,line join=round]
    			(v5) .. controls +(0.45,0.72) and +(-0.45,0.72) .. (v5);
    			
    			\draw[line width=0.8pt,line cap=round,line join=round]
    			(v6) .. controls (6.48,0.56) and (7.52,0.56) .. (v8);
    			
    			\draw[line width=0.8pt,line cap=round,line join=round]
    			(v7) .. controls (7.72,0.92) and (9.30,0.92) .. (v10);

    			%--------------------------------------------------------
    			% arcs below the line
    			%--------------------------------------------------------
    			
    			\draw[line width=0.8pt,line cap=round,line join=round]
    			(v3) .. controls (2.55,-0.52) and (1.45,-0.52) .. (v1);
    			
    			\draw[line width=0.8pt,line cap=round,line join=round]
    			(v8) .. controls (6.75,-1.00) and (4.25,-1.00) .. (v3);
    			
    			\draw[line width=0.8pt,line cap=round,line join=round]
    			(v9) .. controls (8.45,-0.48) and (7.55,-0.48) .. (v7);
    			
    			\draw[line width=0.8pt,line cap=round,line join=round]
    			(v10) .. controls (8.40,-1.42) and (3.60,-1.42) .. (v2);

    			\foreach \i in {1,...,10}{
    				\node[circle,fill=black,inner sep=1.45pt] at (v\i) {};
    				\node[font=\small,below=4pt] at (v\i) {$\i$};
    			}

    		\end{tikzpicture}
    		
    		\caption{The graphic  representation of $\pi=641958(10)372$.}
    		\label{fig:graphic}
    	\end{figure}
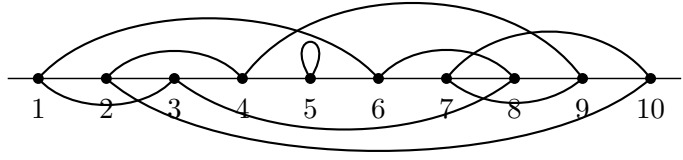

    	Given a Laguerre history $(M,w)\in\L_{n}$ with $M=m_1m_2\ldots m_n$ and $w=w_1w_2\ldots w_n$, we can generate a permutation $\pi=\Gamma(M,w)$ by the following procedure. 
    	 \begin{itemize}
    	 	\item[Step 1.] Place vertices $1,2,\ldots n$ along a horizontal axis. 
    	 	\item[Step 2.] Let $F=\{i\mid m_i=U\,\, \mbox{or}\,\, H_1\}$, $F'=\{i\mid m_i=D \,\, \mbox{or}\,\, H_1\}$, $G=\{i\mid m_i=D\,\, \mbox{or}\,\, H_2\}$ and $G'=\{i\mid m_i=U\,\, \mbox{or}\,\, H_2\}$. 
    	 	\item[Step 3.] Starting from the   \blue{largest} element of $F$, for each $i\in F$,  
    	 	choose the $(w_i+1)$-th \blue{largest} element $j\in F'$ such that $j\geq i$ and $j$ has not yet been chosen. 
    	 	Draw an arc above the horizontal axis connecting vertices $i$ and $j$.
    	 	
    	 	\item[Step 4.] Starting from the \blue{smallest} element of $G$, for each $i\in G$,
    	 	\begin{itemize}
    	 		\item if $m_i=H_2$, choose the $(w_i+1)$-th \blue{smallest} element $j\in G'$ such that $j<i$ and $j$ has not yet been selected. Draw an arc below the horizontal axis connecting vertices $i$ and $j$;
    	 		\item if $m_i = D$, find the element $j\in G'$  such that  
    	 		\begin{itemize}
    	 			\item $j<i$;
    	 			\item 	$j$ has not yet been selected;
    	 			\item  there exists a path of arcs from vertex $i$ to vertex $j$.
    	 			\end{itemize}
    	 		If $w_i=0$, draw an arc below the horizontal axis connecting vertices $i$ and $j$. If $w_i>0$, find the $w_i$-th \blue{smallest} element $k\in G'\setminus\{j\}$ such that $k<i$ and $k$ has not yet been selected. Draw an arc below the horizontal axis connecting vertices $i$ and $k$.
    	 	\end{itemize}

    	 \end{itemize}
    	 
    	 For example, if $(M,w)=(UUH_2H_1UDDD, 01110100)$, then
    	 $F=\{1,2,4,5\}$, $F'=\{4,6,7,8\}$, $G=\{3,6,7,8\}$ and $G'=\{1,2,3,5\}$. The successive construction of $\Gamma(M,w)$ is illustrated in Figure \ref{fig:Gamma}. When $i=6$, we have $m_i=D$. In this case, we first locate the element  $3\in G'$ that has not yet been selected such that there exists a path of arcs from vertex $3$ to vertex $6$.  Since $w_i=1$, we join vertex $1$ and vertex $6$ by an arc below the horizontal axis, as shown in Figure \ref{fig:Gamma}.
    	 
    	 \begin{figure}[htbp]
    	 	\centering
    	 	
    	 	\resizebox{0.90\textwidth}{!}{%
    	 		\begin{tikzpicture}[
    	 			>=Stealth,
    	 			box/.style={
    	 				draw=black!45,
    	 				fill=black!2,
    	 				rounded corners=1pt,
    	 				line width=0.45pt,
    	 				minimum width=5.15cm,
    	 				minimum height=2.55cm
    	 			},
    	 			baseline/.style={
    	 				draw=black,
    	 				line width=0.55pt
    	 			},
    	 			oldedge/.style={
    	 				draw=black,
    	 				line width=0.72pt,
    	 				line cap=round,
    	 				line join=round
    	 			},
    	 			newedge/.style={
    	 				draw=blue!70!violet,
    	 				line width=1.00pt,
    	 				line cap=round,
    	 				line join=round
    	 			},
    	 			point/.style={
    	 				circle,
    	 				draw=black,
    	 				fill=white,
    	 				inner sep=0.95pt
    	 			},
    	 			active/.style={
    	 				circle,
    	 				draw=red!70!black,
    	 				fill=red!70,
    	 				inner sep=1.20pt
    	 			},
    	 			flow/.style={
    	 				->,
    	 				line width=0.72pt
    	 			}
    	 			]
    	 			
    	 			%============================================================
    	 			% Parameters
    	 			%============================================================
    	 			\def\innerscale{0.70}
    	 			
    	 			%============================================================
    	 			% Coordinates of the eight vertices
    	 			%============================================================
    	 			\newcommand{\statecoords}{
    	 				\coordinate (p1) at (-2.03,0);
    	 				\coordinate (p2) at (-1.45,0);
    	 				\coordinate (p3) at (-0.87,0);
    	 				\coordinate (p4) at (-0.29,0);
    	 				\coordinate (p5) at ( 0.29,0);
    	 				\coordinate (p6) at ( 0.87,0);
    	 				\coordinate (p7) at ( 1.45,0);
    	 				\coordinate (p8) at ( 2.03,0);
    	 			}
    	 			
    	 			%============================================================
    	 			% Baseline
    	 			%============================================================
    	 			\newcommand{\drawbaseline}{
    	 				\draw[baseline] (-2.28,0) -- (2.28,0);
    	 			}
    	 			
    	 			%============================================================
    	 			% Vertices and labels
    	 			%============================================================
    	 			\newcommand{\drawstatevertices}{
    	 				\foreach \i in {1,...,8}{
    	 					\node[point] at (p\i) {};
    	 					\node[font=\scriptsize,below=2.1pt] at (p\i) {$\i$};
    	 				}
    	 			}
    	 			
    	 			%============================================================
    	 			% Box positions
    	 			%============================================================
    	 			\node[box] (B1) at (0,0) {};
    	 			\node[box] (B2) at (5.90,0) {};
    	 			\node[box] (B3) at (11.80,0) {};
    	 			
    	 			\node[box] (B4) at (0,-3.35) {};
    	 			\node[box] (B5) at (5.90,-3.35) {};
    	 			\node[box] (B6) at (11.80,-3.35) {};
    	 			
    	 			\node[box] (B7) at (2.95,-6.70) {};
    	 			\node[box] (B8) at (8.85,-6.70) {};
    	 			
    	 			%============================================================
    	 			% F1
    	 			%============================================================
    	 			\begin{scope}[shift={(B1.center)}]
    	 				\begin{scope}[shift={(0,0.25)},scale=\innerscale]
    	 					
    	 					\statecoords
    	 					\drawbaseline
    	 					
    	 					% new upper edge 5--8
    	 					\draw[newedge]
    	 					(p5) .. controls +(0.52,1.12) and +(-0.52,1.12) .. (p8);
    	 					
    	 					\drawstatevertices
    	 					\node[active] at (p5) {};
    	 					
    	 				\end{scope}
    	 				
    	 				\node[font=\small] at (0,-0.90)
    	 				{$i=5$};
    	 			\end{scope}
    	 			
    	 			%============================================================
    	 			% F2
    	 			%============================================================
    	 			\begin{scope}[shift={(B2.center)}]
    	 				\begin{scope}[shift={(0,0.25)},scale=\innerscale]
    	 					
    	 					\statecoords
    	 					\drawbaseline
    	 					
    	 					% old edge
    	 					\draw[oldedge]
    	 					(p5) .. controls +(0.52,1.12) and +(-0.52,1.12) .. (p8);
    	 					
    	 					% new edge 4--6
    	 					\draw[newedge]
    	 					(p4) .. controls +(0.42,0.72) and +(-0.42,0.72) .. (p6);
    	 					
    	 					\drawstatevertices
    	 					\node[active] at (p4) {};
    	 					
    	 				\end{scope}
    	 				
    	 				\node[font=\small] at (0,-0.90)
    	 				{$i=4$};
    	 			\end{scope}
    	 			
    	 			%============================================================
    	 			% F3
    	 			%============================================================
    	 			\begin{scope}[shift={(B3.center)}]
    	 				\begin{scope}[shift={(0,0.25)},scale=\innerscale]
    	 					
    	 					\statecoords
    	 					\drawbaseline
    	 					
    	 					\draw[oldedge]
    	 					(p5) .. controls +(0.52,1.12) and +(-0.52,1.12) .. (p8);
    	 					
    	 					\draw[oldedge]
    	 					(p4) .. controls +(0.42,0.72) and +(-0.42,0.72) .. (p6);
    	 					
    	 					% new edge 2--4
    	 					\draw[newedge]
    	 					(p2) .. controls +(0.42,0.68) and +(-0.42,0.68) .. (p4);
    	 					
    	 					\drawstatevertices
    	 					\node[active] at (p2) {};
    	 					
    	 				\end{scope}
    	 				
    	 				\node[font=\small] at (0,-0.90)
    	 				{$i=2$};
    	 			\end{scope}
    	 			
    	 			%============================================================
    	 			% F4
    	 			%============================================================
    	 			\begin{scope}[shift={(B4.center)}]
    	 				\begin{scope}[shift={(0,0.25)},scale=\innerscale]
    	 					
    	 					\statecoords
    	 					\drawbaseline
    	 					
    	 					\draw[oldedge]
    	 					(p5) .. controls +(0.52,1.12) and +(-0.52,1.12) .. (p8);
    	 					
    	 					\draw[oldedge]
    	 					(p4) .. controls +(0.42,0.72) and +(-0.42,0.72) .. (p6);
    	 					
    	 					\draw[oldedge]
    	 					(p2) .. controls +(0.42,0.68) and +(-0.42,0.68) .. (p4);
    	 					
    	 					% new edge 1--7
    	 					\draw[newedge]
    	 					(p1) .. controls +(1.15,1.72) and +(-1.15,1.72) .. (p7);
    	 					
    	 					\drawstatevertices
    	 					\node[active] at (p1) {};
    	 					
    	 				\end{scope}
    	 				
    	 				\node[font=\small] at (0,-0.90)
    	 				{$i=1$};
    	 			\end{scope}
    	 			
    	 			%============================================================
    	 			% F5
    	 			% i = 3, new lower edge 2--3
    	 			%============================================================
    	 			\begin{scope}[shift={(B5.center)}]
    	 				\begin{scope}[shift={(0,0.25)},scale=\innerscale]
    	 					
    	 					\statecoords
    	 					\drawbaseline
    	 					
    	 					% old upper edges
    	 					\draw[oldedge]
    	 					(p5) .. controls +(0.52,1.12) and +(-0.52,1.12) .. (p8);
    	 					
    	 					\draw[oldedge]
    	 					(p4) .. controls +(0.42,0.72) and +(-0.42,0.72) .. (p6);
    	 					
    	 					\draw[oldedge]
    	 					(p2) .. controls +(0.42,0.68) and +(-0.42,0.68) .. (p4);
    	 					
    	 					\draw[oldedge]
    	 					(p1) .. controls +(1.15,1.72) and +(-1.15,1.72) .. (p7);
    	 					
    	 					% new lower edge 2--3
    	 					\draw[newedge]
    	 					(p2) .. controls +(0.20,-0.42) and +(-0.20,-0.42) .. (p3);
    	 					
    	 					\drawstatevertices
    	 					\node[active] at (p3) {};
    	 					
    	 				\end{scope}
    	 				
    	 				\node[font=\small] at (0,-0.90)
    	 				{$i=3$};
    	 			\end{scope}
    	 			
    	 			%============================================================
    	 			% F6
    	 			%============================================================
    	 			\begin{scope}[shift={(B6.center)}]
    	 				\begin{scope}[shift={(0,0.25)},scale=\innerscale]
    	 					
    	 					\statecoords
    	 					\drawbaseline
    	 					
    	 					% old upper edges
    	 					\draw[oldedge]
    	 					(p5) .. controls +(0.52,1.12) and +(-0.52,1.12) .. (p8);
    	 					
    	 					\draw[oldedge]
    	 					(p4) .. controls +(0.42,0.72) and +(-0.42,0.72) .. (p6);
    	 					
    	 					\draw[oldedge]
    	 					(p2) .. controls +(0.42,0.68) and +(-0.42,0.68) .. (p4);
    	 					
    	 					\draw[oldedge]
    	 					(p1) .. controls +(1.15,1.72) and +(-1.15,1.72) .. (p7);
    	 					
    	 					% old lower edge 2--3
    	 					\draw[oldedge]
    	 					(p2) .. controls +(0.20,-0.42) and +(-0.20,-0.42) .. (p3);
    	 					
    	 					% new lower edge 1--6
    	 					\draw[newedge]
    	 					(p1) .. controls +(1.10,-1.48) and +(-0.92,-1.48) .. (p6);
    	 					
    	 					\drawstatevertices
    	 					\node[active] at (p6) {};
    	 					
    	 				\end{scope}
    	 				
    	 				\node[font=\small] at (0,-0.90)
    	 				{$i=6$};
    	 			\end{scope}
    	 			
    	 			%============================================================
    	 			% F7
    	 			% i = 7, new lower edge 3--7
    	 			%============================================================
    	 			\begin{scope}[shift={(B7.center)}]
    	 				\begin{scope}[shift={(0,0.25)},scale=\innerscale]
    	 					
    	 					\statecoords
    	 					\drawbaseline
    	 					
    	 					% old upper edges
    	 					\draw[oldedge]
    	 					(p5) .. controls +(0.52,1.12) and +(-0.52,1.12) .. (p8);
    	 					
    	 					\draw[oldedge]
    	 					(p4) .. controls +(0.42,0.72) and +(-0.42,0.72) .. (p6);
    	 					
    	 					\draw[oldedge]
    	 					(p2) .. controls +(0.42,0.68) and +(-0.42,0.68) .. (p4);
    	 					
    	 					\draw[oldedge]
    	 					(p1) .. controls +(1.15,1.72) and +(-1.15,1.72) .. (p7);
    	 					
    	 					% old lower edge 2--3
    	 					\draw[oldedge]
    	 					(p2) .. controls +(0.20,-0.42) and +(-0.20,-0.42) .. (p3);
    	 					
    	 					% old lower edge 1--6
    	 					\draw[oldedge]
    	 					(p1) .. controls +(1.10,-1.48) and +(-0.92,-1.48) .. (p6);
    	 					
    	 					% new lower edge 3--7
    	 					\draw[newedge]
    	 					(p3) .. controls +(0.98,-1.18) and +(-0.98,-1.18) .. (p7);
    	 					
    	 					\drawstatevertices
    	 					\node[active] at (p7) {};
    	 					
    	 				\end{scope}
    	 				
    	 				\node[font=\small] at (0,-0.90)
    	 				{$i=7$};
    	 			\end{scope}
    	 			
    	 			%============================================================
    	 			% F8
    	 			%============================================================
    	 			\begin{scope}[shift={(B8.center)}]
    	 				\begin{scope}[shift={(0,0.25)},scale=\innerscale]
    	 					
    	 					\statecoords
    	 					\drawbaseline
    	 					
    	 					% old upper edges
    	 					\draw[oldedge]
    	 					(p5) .. controls +(0.52,1.12) and +(-0.52,1.12) .. (p8);
    	 					
    	 					\draw[oldedge]
    	 					(p4) .. controls +(0.42,0.72) and +(-0.42,0.72) .. (p6);
    	 					
    	 					\draw[oldedge]
    	 					(p2) .. controls +(0.42,0.68) and +(-0.42,0.68) .. (p4);
    	 					
    	 					\draw[oldedge]
    	 					(p1) .. controls +(1.15,1.72) and +(-1.15,1.72) .. (p7);
    	 					
    	 					% old lower edge 2--3
    	 					\draw[oldedge]
    	 					(p2) .. controls +(0.20,-0.42) and +(-0.20,-0.42) .. (p3);
    	 					
    	 					% old lower edge 1--6
    	 					\draw[oldedge]
    	 					(p1) .. controls +(1.10,-1.48) and +(-0.92,-1.48) .. (p6);
    	 					
    	 					% old lower edge 3--7
    	 					\draw[oldedge]
    	 					(p3) .. controls +(0.98,-1.18) and +(-0.98,-1.18) .. (p7);
    	 					
    	 					% new lower edge 5--8
    	 					\draw[newedge]
    	 					(p5) .. controls +(0.64,-0.98) and +(-0.64,-0.98) .. (p8);
    	 					
    	 					\drawstatevertices
    	 					\node[active] at (p8) {};
    	 					
    	 				\end{scope}
    	 				
    	 				\node[font=\small] at (0,-0.90)
    	 				{$i=8$};
    	 			\end{scope}
    	 			
    	 			%============================================================
    	 			% Arrows between figures
    	 			%============================================================
    	 			
    	 			% first row
    	 			\draw[flow]
    	 			($(B1.east)+(0.14,0)$)
    	 			--
    	 			($(B2.west)+(-0.14,0)$);
    	 			
    	 			\draw[flow]
    	 			($(B2.east)+(0.14,0)$)
    	 			--
    	 			($(B3.west)+(-0.14,0)$);
    	 			
    	 			% arrow before second row
    	 			\draw[flow]
    	 			($(B4.west)+(-0.75,0)$)
    	 			--
    	 			($(B4.west)+(-0.18,0)$);
    	 			
    	 			% second row
    	 			\draw[flow]
    	 			($(B4.east)+(0.14,0)$)
    	 			--
    	 			($(B5.west)+(-0.14,0)$);
    	 			
    	 			\draw[flow]
    	 			($(B5.east)+(0.14,0)$)
    	 			--
    	 			($(B6.west)+(-0.14,0)$);
    	 			
    	 			% arrow before third row
    	 			\draw[flow]
    	 			($(B7.west)+(-0.75,0)$)
    	 			--
    	 			($(B7.west)+(-0.18,0)$);
    	 			
    	 			% third row
    	 			\draw[flow]
    	 			($(B7.east)+(0.14,0)$)
    	 			--
    	 			($(B8.west)+(-0.14,0)$);
    	 			
    	 		\end{tikzpicture}%
    	 	}
    	 	
    	 	\caption{An illustration of the map $\Gamma$.}
    	 	\label{fig:Gamma}
    	 	
    	 \end{figure}
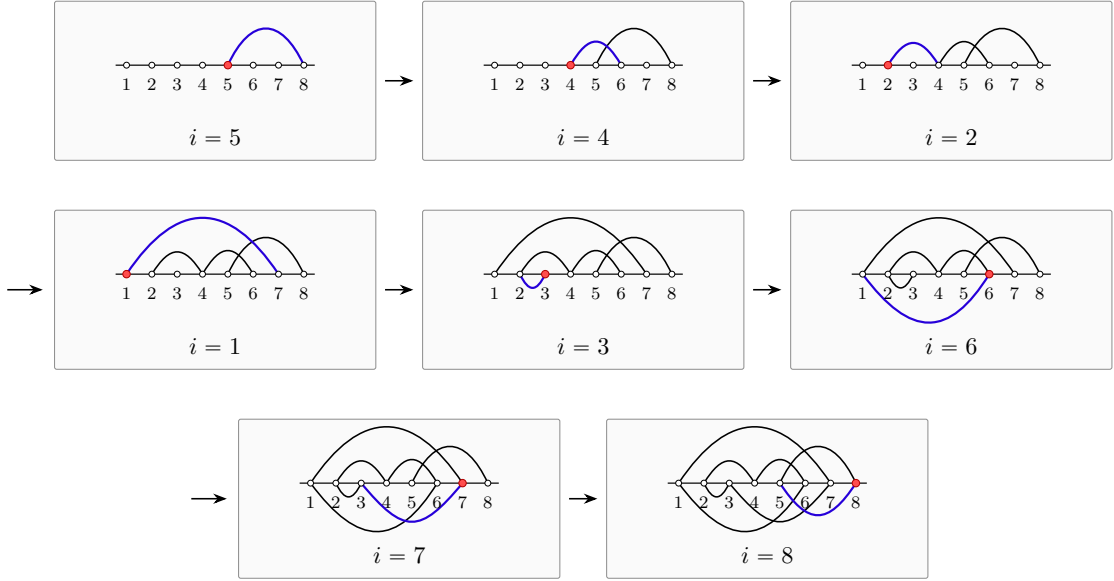

    	 \begin{theorem}\label{thm-Gamma}
    	 	For $n\geq 1$, the map $\Gamma$ induces a bijection between $\mathcal{L}_n$ and $\mathfrak{S}_n$. Moreover, for any Laguerre history $(M,w)\in\mathcal{L}_n$, where $M=m_1m_2\cdots m_n$ and $w=w_1w_2\cdots w_n$, the corresponding permutation $\pi=\Gamma(M,w)$ satisfies the following identites: 
    	 	\begin{equation}\label{eq-pcyc}
    	 		\pcyc(\pi)=\#\, \{k\mid (m_k, w_k)=(D, 0)\},
    	 	\end{equation}	 
    	 	and 
    	 	\begin{equation}\label{eq-Rec-Gamma}
    	 		\Rec(\pi)=\{k\mid (m_k, w_k)=(H_1, 0)\}| \cup\{k\mid (m_k, w_k)=(U, 0)\}.
    	 	\end{equation}
    	 	\end{theorem}
    	 	\begin{proof}
    	 		First we aim to show that the map $\Gamma$ is an injection. We retain the notations from the definition of $\Gamma$. To prove the injectivity of $\Gamma$, it suffices to demonstrate how one may recover $(M,w)$ from $\pi$.
    	 		
    	 		By the construction of $\Gamma(M,w)$, one readily verifies that $\pi(i)\geq i$ for each $i\in F$ and $\pi(i)<i$ for eachy $i\in G$. Moreover, $\pi^{-1}(i)\leq i$ for each $i\in F'$ and $\pi^{-1}(i)>i$ for each $i\in G'$. This yields that
    	 	\begin{equation}\label{eq-steps}
    	 		m_i=
    	 		\begin{cases}
    	 			U & \text{if } \pi(i)>i<\pi^{-1}(i), \\
    	 			D & \text{if } \pi(i)<i>pi^{-1}(i), \\
    	 			H_1 & \text{if } \pi(i)\geq i\geq \pi^{-1}(i), \\
    	 			H_2 & \text{if } \pi(i)<i<\pi^{-1}(i),
    	 		\end{cases}
    	 	\end{equation}
    	 		Again from the construction of $\Gamma(M,w)$, one readily sees that $w_i=\nest_i(\pi)$ whenever $m_i\neq D$. If $m_i=D$, then $i\in G\cap F'$. In this case, we locate the element $j\in G'$ that has not yet been selected such that vertices $i$ and $j$ are joined by a path of arcs.
    	 		
    	 		If $w_i=0$, we draw an arc below the horizontal axis connecting vertices $i$ and $j$, and $i$ turns out to be the largest element of a cycle of $\pi$.
    	 		
    	 		 If $w_i>0$, select the $w_i$-th smallest element $k\in G'\setminus\{j\}$ satisfying that $k<i$ and $k$ has not yet been selected, then draw an arc below the horizontal axis connecting vertices $i$ and $k$. Clearly,  $i$ is not the largest element of a cycle of $\pi$. Moreover, it  is  plain to check that 
    	 		\[
    	 		w_i=\nest_i(\pi)+\mathbf{1}_{j>\pi(i)},
    	 		\]
    	 		where $\mathbf{1}_{A}=1$ if statement $A$ holds, and $\mathbf{1}_{A}=0$ otherwise.
    	 		
    	 		From the above analysis, we may recover $(M,w)\in \mathcal{L}_n$ with $M=m_1m_2\cdots m_n$ and $w=w_1w_2\cdots w_n$, from a permutation $\pi$ via the following procedure.
    	 		For each $i\in [n]$, let
    	 		\[
    	 		m_i=
    	 		\begin{cases}
    	 			U & \text{if } \pi(i)>i<\pi^{-1}(i), \\
    	 			D & \text{if } \pi(i)<i>\pi^{-1}(i), \\
    	 			H_1 & \text{if } \pi(i)\geq i\geq \pi^{-1}(i), \\
    	 			H_2 & \text{if } \pi(i)<i<\pi^{-1}(i),
    	 		\end{cases}
    	 		\]
    	 		and the weight $w_i$ is determined by the following rules:
    	 		\begin{itemize}
    	 			\item If $m_i=D$ and $i$ is the largest element of a cycle in $\pi$, then set $w_i=0$;
    	 			\item If $m_i=D$ and $i$ is not the largest element of a cycle in $\pi$, find the vertex $j$ such that $\pi^{-1}(j)>i>j$ and vertices $i$ and $j$ are joined by a path of arcs. Let
    	 			\[
    	 			w_i=\nest_i(\pi)+\mathbf{1}_{j>\pi(i)}.
    	 			\]
    	 			\item If $m_i\neq D$, set $w_i=\nest_i(\pi)$.
    	 		\end{itemize}
    	 		Therefore the map $\Gamma$ is injective. It remains to establish surjectivity of $\Gamma$. This follows from the equality $|\mathcal{L}_n|=|\mathfrak{S}_n|$.
    	 		
    	 		Now we proceed to prove (\ref{eq-pcyc}) and  (\ref{eq-Rec-Gamma}). By observation \ref{ob0},  it follows   that $i\in \Rec(\pi)$ if and only if $\pi(i)\geq i$ and $\nest_i(\pi)=0$. This implies that $i\in \Rec(\pi)$ if and only  $m_i=U$ or $H_1$ and $w_i=0$, completing the proof of (\ref{eq-Rec-Gamma}). 
    	 		
    	 		It is easily seen that if  $i$ is the largest element of a pure cycle  of $\pi$, then  $i$ must be   a cycle peak. This implies that $m_i=D$. Recall that we have concluded that in the case $m_i = D$, we have $w_i = 0$ if and only if $i$ is the largest element of a pure cycle in $\pi$.  This completes the proof of  (\ref{eq-pcyc}). 
    	\end{proof}
    	 
    	  \begin{theorem}\label{thm-Theta}
    	 	For $n\geq 1$, the map $\Theta:=\Gamma\circ \Psi_{FZ}$ is a bijection from  $\mathfrak{S}_n$   onto  itself such that for any permutation  $\pi\in \mathfrak{S}_n$, we have  
    	 	$$
    	 	(\des_2, \ear)\pi= (\des_2, \pcyc)\Theta(\pi).
    	 	$$
    	 \end{theorem}
    	 \begin{proof}
    	 	In view of Theorem \ref{thm-Gamma}, the map $\Theta:=\Gamma\circ \Psi_{FZ}$ is  a bijection from  $\mathfrak{S}_n$ to itself. 
    	 	Given a permutation $\pi\in \mathfrak{S}_n$, let $(M,w)=\Psi_{FZ}(\pi)$ and $\sigma=\Gamma(M,w)$. 
    	 	 Invoking (\ref{eq-ear-Mot}), (\ref{eq-Rec-Mot}), (\ref{eq-pcyc}) and (\ref{eq-Rec-Gamma}), we deduce that
    	 	$$
    	 	\ear(\pi)=\#\, \{k\mid (m_k, w_k)=(D, 0)\}=\pcyc(\sigma)
    	 	$$
    	 	and 
    	 	\begin{equation}\label{eq-Rec-Theta}
    	 	\Rec(\pi)=\{k\mid (m_k, w_k)=(H_1, 0)\} \cup\{k\mid (m_k, w_k)=(U, 0)\}=\Rec(\sigma).
    	 	\end{equation}
    Combining Observation \ref{ob0} and (\ref{eq-Rec-Theta}), it follows that
    $\Des_2(\pi)=\Des_2(\sigma)=\Des_2(\Theta(\pi))$ and hence $\des_2(\pi)=\des_2(\Theta(\pi))$ as desired, completing the proof.
    	 	\end{proof}
    	 	
    	 	 Given a  Laguerre history  $(M,w)\in \mathcal{L}_n$ with $M=m_1m_2\ldots m_n$ and $w=w_1w_2\ldots w_n$, we can constrcut a Laguerre history  $\zeta(M,w)=(M,w')\in \mathcal{L}_n$  with $w'=w'_1w'_2\ldots w'_n$ by letting
    	 	$$
    	 	w'_i=	\begin{cases}
    	 		h_i(M)-1-w_i&\quad \text{if $m_i=D$ },\\
    	 		w_i&\quad \text{otherwise}.\\ 
    	 	\end{cases}
    	 	$$
    	 	Clearly, the map $\zeta$ is an involution on $\mathcal{L}_n$.

    	 	  \begin{theorem}\label{thm-Omega}
    	 		For $n\geq 1$, the map $\Omega:= \Psi^{-1}_{FZ}\circ\zeta\circ \Psi_{FZ}$ is an involution on  $\mathfrak{S}_n$ such that for any permutation  $\pi\in \mathfrak{S}_n$, we have  
    	 		$$
    	 		(\des_2, \ear)\pi= (\des_2, \pdrop)\Omega(\pi).
    	 		$$
    	 	\end{theorem}
    	 	
    	 	\begin{proof}
    	 		Since   $\zeta$ is an involution and $\Psi_{FZ}$ is a bijection,  the map $\Omega:= \Psi^{-1}_{FZ}\circ\zeta\circ \Psi_{FZ}$ is  an involution on $\mathfrak{S}_n$.  Given a permutation $\pi\in \mathfrak{S}_n$, let $(M,w)=\Psi_{FZ}(\pi)$, $\zeta(M,w)=(M,w')$ and  $\sigma=\Psi^{-1}_{FZ}(M,w')$, where $M=m_1m_2\ldots m_n$, $w=w_1w_2\ldots w_n$  and $w'=w'_1w'_2\ldots w'_n$. 
    	 		Invoking (\ref{eq-ear-Mot})-(\ref{eq-Rec-Mot}),  we deduce that
    	 		$$
    	 		\ear(\pi)=\#\, \{k\mid (m_k, w_k)=(D, 0)\}= \#\, \{k\mid (m_k, w'_k)=(D, h_i(M)-1)\}=\pdrop(\sigma)
    	 		$$
    	 		and 
    	 		\begin{equation}\label{eq-Rec-Omega}
    	 			\begin{array}{lll}
    	 			\Rec(\pi)&=&\{k\mid (m_k, w_k)=(H_1, 0)\}\cup\{k\mid (m_k, w_k)=(U, 0)\}\\
    	 			&=&\{k\mid (m_k, w'_k)=(H_1, 0)\} \cup\{k\mid (m_k, w'_k)=(U, 0)\}=\Rec(\sigma).
    	 			\end{array}
    	 		\end{equation}
    	 		Combining Observation \ref{ob0} and (\ref{eq-Rec-Omega}), it follows that
    	 		$\Des_2(\pi)=\Des_2(\sigma)=\Des_2(\Omega(\pi))$ and hence $\des_2(\pi)=\des_2(\Omega(\pi))$ as desired, completing the proof.
    	 		\end{proof}

    	 	\noindent{\bf Proof of Theorem \ref{thm-equi}.} 
    	 		Combining Theorems \ref{thm-Theta} and \ref{thm-Omega}, we establish the equidistribution of the three bistatistics
    	 		$(\des_2, \ear)$, $(\des_2, \pdrop)$ and $(\des_2, \pcyc)$ on $\mathfrak{S}_n$. Combined with Theorem \ref{thm-Psi}, this yields the desired equidistribution, completing the proof. \qed

   \section*{Acknowledgments}  
   The work  was supported by
   the National Natural
   Science Foundation of China grants 12471318.

	%%%%%%%%%%%%%%%%%%%%%%%%%%%

\end{document}